\documentclass[english,12pt]{article}

\pdfoutput=1

\usepackage[T1]{fontenc}
\usepackage[latin9]{inputenc}
\usepackage{amssymb}
\usepackage{babel}
\usepackage{amsthm}
\usepackage{verbatim}
\usepackage{amsrefs}
\usepackage{amsmath}
\usepackage{listings}
\usepackage{fullpage}
\usepackage{graphicx}
\usepackage{qtree}
\usepackage[usenames,dvipsnames]{color}
\usepackage{array}
\usepackage{enumerate}
\usepackage{polynom}
\usepackage{tikz}
\usepackage{fancyhdr}
\usepackage{etoolbox}
\usepackage{algorithm}
\usepackage{algorithmic}

\newtheorem{defin}{Definition}

\newtheorem{theorem}{Theorem}

\newcommand{\p}[1]{\left(#1\right)}

\newcommand{\quot}[1]{``#1''}

\newcommand{\namehead}[3]{
\lstset{breaklines=true, morecomment=[l]{//}, frame=single, showstringspaces=false, numbers=left}
\begin{flushright}
Nathan Fox\\
#2\\
#3\\
\end{flushright}
\ifstrequal{#1}{.}{}{
\begin{center}
{\Large Homework #1}
\end{center}}
}

\usepackage{placeins}

\newcommand{\seq}{\p}
\newcommand{\Mf}{M}
\newcommand{\Rh}{R}

\begin{document}
%
%
\title{Linear Recurrent Subsequences of Meta-Fibonacci Sequences}
\author{Nathan Fox\footnote{Department of Mathematics, Rutgers University, Piscataway, New Jersey,
\texttt{fox@math.rutgers.edu}
}}
\date{}

\maketitle

\begin{abstract}
In a recent paper, Frank Ruskey asked whether every linear recurrent sequence can occur in some solution of a meta-Fibonacci sequence.  In this paper, we answer his question in the affirmative for recurrences with positive coefficients.
\end{abstract}

\section{Introduction}
Various authors use various definitions of meta-Fibonacci sequences, which were first introduced by Douglas Hofstadter~\cite[pp.\ 137-138]{geb}.  In this paper, we will use the following definition.
\begin{defin}
A \emph{meta-Fibonacci sequence} is a sequence of integers
$\seq{q_n}_{n\geq n_0}$
that eventually satisfy a recurrence of the form
\[
\Mf\p{n}=\sum_{i=1}^kb_i\Mf\p{n-\Mf\p{n-i}}
\]
for some fixed positive integers $k$ and $b_i$.  (Typically, $n_0$ is $0$ or $1$.)
\end{defin}
Other authors allow some variations of this definition~\cite{tanny}, but all would agree that any sequence described this way is in fact meta-Fibonacci.


The original and most well-known meta-Fibonacci sequence is Hofstadter's $Q$-sequence.  His sequence (with $n_0=1$) is defined by $Q\p{1}=1$, $Q\p{2}=1$, and for $n\geq3$, $Q\p{n}=Q\p{n-Q\p{n-1}}+Q\p{n-Q\p{n-2}}$~\cite{geb}.  A primary question asked about meta-Fibonacci sequences (and famously still open for Hofstadter's $Q$-sequence) is whether they are defined for all $n$.  In general, it is conceivable that one of two undesirable things could happen when trying to evaluate $\Mf\p{n}$ for a meta-Fibonacci sequence:
\begin{itemize}
\item When trying to evaluate $\Mf\p{n}$, $\Mf\p{m}$ needs to be evaluated for some $m<n_0$.  (For example, this would happen when evaluating $Q\p{3}$ if the initial conditions to Hofstadter's recurrence were $Q\p{1}=1$, $Q\p{2}=3$.)
\item When trying to evaluate $\Mf\p{n}$, $\Mf\p{m}$ needs to be evaluated for some $m\geq n$.  (For example, this would happen when evaluating $Q\p{3}$ if the initial conditions to Hofstadter's recurrence were $Q\p{1}=1$, $Q\p{2}=0$.)
\end{itemize}
Often, authors use the convention that, when dealing with meta-Fibonacci sequences, $\Mf\p{n}=0$ whenever $n<n_0$.  We will use this convention going forward.  This causes the first potential issue to go away, though the second one remains nonrecoverable.

Hofstadter's original sequence appears to behave quite chaotically.  But, related sequences have been found with much more predictable behavior.  For example, Tanny was able to slightly modify the recurrence to yield a better-behaved but similar-looking sequence~\cite{tanny}.  Also, Golomb discovered that, under Hofstadter's original recurrence, changing the initial condition to $Q\p{1}=3$, $Q\p{2}=2$, and $Q\p{3}=1$ yields a quasilinear solution of period three~\cite{golomb}.  In a similar vein, Ruskey discovered (using the convention that evaluating at a negative index gives zero) a 
solution to Hofstadter's $Q$-recurrence that includes the Fibonacci numbers starting from $5$ as every third term~\cite{rusk}.  At the end of his paper, Ruskey asks whether every linear recurrent sequence exists as an equally-spaced subsequence of a solution to some meta-Fibonacci recurrence.  In this paper, we answer this question positively for recurrences with positive coefficients.  
In particular, our proof is constructive.
\section{The Construction}
We will prove this main theorem:
\begin{theorem}\label{thm:main}
Let $\seq{a_n}_{n\geq0}$ be a sequence of positive integers satisfying the recurrence
\[
a_n=\sum_{i=1}^kb_ia_{n-i},
\]
for some positive integer $k$ and nonnegative integers $b_1,b_2,\ldots,b_k$ whose sum is at least $2$.  Then, there is a 
meta-Fibonacci sequence $\seq{q_n}_{n\geq0}$
and a positive integer $s$ such that $q_{sn}=a_n$ for all $n\geq0$.  (We will call the number $s$ the \emph{quasi-period} of the sequence $\seq{q_n}$.)
\end{theorem}

\begin{proof}
Let $\seq{a_n}_{n\geq0}$ be a sequence of positive integers satisfying the recurrence
\[
a_n=\sum_{i=1}^kb_ia_{n-i},
\]
for some positive integer $k$ and nonnegative integers $b_1,b_2,\ldots,b_k$ whose sum is at least $2$.  For each $r$ from $0$ to $k-1$, define the sequence $\seq{a^{\p{r}}_n}_{n\geq1}$ as
\[
a^{\p{r}}_n=\sum_{i=1}^rb_ia^{\p{r}}_{n-k-i+r}+\sum_{i=r+1}^kb_ia^{\p{r}}_{n-i+r}
\]
with $a^{\p{r}}_i=a_i$ for $i\leq k$.  Notice that $a^{\p{0}}_n=a_n$, and the other sequences satisfy similar recurrences with the coefficients cycled.
Since
the coefficients of the recurrences are nonnegative and sum to at least $2$, the sequences $\seq{a^{\p{r}}_n}_{n\geq1}$ exhibit superlinear growth for all $r$.

Now, define the sequence $\seq{q_n}_{n\geq1}$ as follows:
\[
\begin{cases}
q_{2mk+2j}=a^{\p{j}}_m & 0\leq j<k\\
q_{2mk+2j+1}=2k\p{k-j} & 0\leq j<k,
\end{cases}
\]
We claim that $\seq{q_n}_{n\geq1}$ eventually satisfies the meta-Fibonacci recurrence
\[
\Mf_a\p{n}=\Mf_a\p{n-\Mf_a\p{n-2}}+\sum_{i=1}^kb_i\Mf_a\p{n-\Mf_a\p{n-\p{2i-1}}}.
\]
Notice that this will imply the desired result, since
the quasi-period will be $2k$.
Let $h$ be an integer satisfying all of the following constraints:
\begin{itemize}
\item $h\geq2k-1$
\item $h\geq2$
\item For all $r$, whenever $m\geq h$, $a^{\p{r}}_{m-1}\geq 2\p{m+1}k$.
\end{itemize}
We define a function $\Rh$ as follows:
\[
\Rh\p{n}=\begin{cases}
q_n & n\leq h\\
\Rh\p{n}=\displaystyle{\Rh\p{n-\Rh\p{n-2}}+\sum_{i=1}^kb_i\Rh\p{n-\Rh\p{n-\p{2i-1}}}} & n>h.
\end{cases}
\]
In other words, $\Rh$ eventually satisfies the recurrence $\Mf_a$, and it has an initial condition of length $h$ that matches $\seq{q_n}$.  The first two conditions on $h$ are required to make $\Rh$ well-defined.  Since all the linear recurrent sequences under consideration grow superlinearly, the third condition will be satisfied by all sufficiently large numbers.  Hence, such an $h$ exists, and all larger values would also be valid choices for $h$.

We wish to show that
$R\p{n}=q_n$ for all $n$.
We will proceed by induction on $n$.
The base case is covered by the fact that
$R\p{n}$ is defined to equal $q_n$ for $n\leq h$.
So, we will show that, for $n>h$, $R\p{n}=q_n$ under the assumption that $R\p{p}=q_p$ for all $1\leq p<n$.
For this, we will split into two cases:
\begin{description}
\item[$n$ is odd:] 
Since $n$ is odd, it is of the form $2mk+2j+1$ for some $m\geq0$ and some $0\leq j<k$.
By our choice of $h$, $a^{\p{r}}_{m-1}\geq 2\p{m+1}k$ and $a^{\p{r}}_{m}\geq 2\p{m+2}k$.  In particular, both of these are greater than $2mk+2j+1$.  Using this fact, we have
{\allowdisplaybreaks\begin{align*}
\Rh\p{n}&=\Rh\p{n-\Rh\p{n-2}}+\sum_{i=1}^kb_i\Rh\p{n-\Rh\p{n-\p{2i-1}}}\\
&=\Rh\p{n-q_{n-2}}+\sum_{i=1}^kb_i\Rh\p{n-q_{n-\p{2i-1}}}\\
&=\Rh\p{2mk+2j+1-q_{2mk+2j-1}}+\sum_{i=1}^kb_i\Rh\p{2mk+2j+1-q_{2mk+2j+1-\p{2i-1}}}\\
&=\Rh\p{2mk+2j+1-q_{2mk+2\p{j-1}+1}}+\sum_{i=1}^kb_i\Rh\p{2mk+2j+1-q_{2mk+2\p{j-i+1}}}\\
&=\Rh\p{2mk+2j+1-q_{2mk+2\p{j-1}+1}}+\sum_{i=1}^{j+1}b_i\Rh\p{2mk+2j+1-a^{\p{j-i+1}}_m}\\
&\hspace{0.2in}+\sum_{i=j+2}^kb_i\Rh\p{2mk+2j+1-a^{\p{k+j-i+1}}_{m-1}}\\
&=\Rh\p{2\p{m-k+j}k+2j+1}+\sum_{i=1}^{j+1}b_i\cdot0+\sum_{i=j+2}^kb_i\cdot0\\
&=q_{2\p{m-k+j}k+2j+1}\\
&=2k\p{k-j}\\
&=q_{2mk+2j+1}\\
&=q_n,
\end{align*}}
as required.
\item[$n$ is even:] 
Since $n$ is even, it is of the form $2mk+2j$ for some $m\geq0$ and some $0\leq j< k$.
By our choice of $h$, $a^{\p{r}}_{m-1}\geq 2\p{m+1}k$ and $a^{\p{r}}_{m}\geq 2\p{m+2}k$.  In particular, both of these are greater than $2mk+2j$.  Using this fact, we have
{\allowdisplaybreaks
\begin{align*}
\Rh\p{n}&=\Rh\p{n-\Rh\p{n-2}}+\sum_{i=1}^kb_i\Rh\p{n-\Rh\p{n-\p{2i-1}}}\\
&=\Rh\p{n-q_{n-2}}+\sum_{i=1}^kb_i\Rh\p{n-q_{n-\p{2i-1}}}\\
&=\Rh\p{2mk+2j-q_{2mk+2j-2}}+\sum_{i=1}^kb_i\Rh\p{2mk+2j-q_{2mk+2j-\p{2i-1}}}\\
&=\Rh\p{2mk+2j-q_{2mk+2\p{j-1}}}+\sum_{i=1}^kb_i\Rh\p{2mk+2j-q_{2mk+2\p{j-i}+1}}
\end{align*}}
If $j=0$, then we will have $q_{2mk+2\p{j-1}}=a^{\p{k-1}}_{m-1}$.  Otherwise, we will have $q_{2mk+2\p{j-1}}=a^{\p{j-1}}_m$.  In either case,
we have $\Rh\p{2mk+2j-q_{2mk+2\p{j-1}}}=0$.  So,
{\allowdisplaybreaks
\begin{align*}
\Rh\p{n}&=0+\sum_{i=1}^kb_i\Rh\p{2mk+2j+1-q_{2mk+2\p{j-i}+1}}\\
&=\sum_{i=1}^jb_i\Rh\p{2mk+2j-2k\p{k-\p{j-i}}}\\
&\hspace{0.2in}+\sum_{i=j+1}^kb_i\Rh\p{2mk+2j-2k\p{k-\p{k+j-i}}}\\
&=\sum_{i=1}^jb_i\Rh\p{2mk+2j-2k\p{k-j+i}}+\sum_{i=j+1}^kb_i\Rh\p{2mk+2j-2k\p{i-j}}\\
&=\sum_{i=1}^jb_i\Rh\p{2\p{m-k-i+j}k+2j}+\sum_{i=j+1}^kb_i\Rh\p{2\p{m-i+j}k+2j}\\
&=\sum_{i=1}^jb_iq_{2\p{m-k-i+j}k+2j}+\sum_{i=j+1}^kb_iq_{2\p{m-i+j}k+2j}\\
&=\sum_{i=1}^jb_ia^{\p{j}}_{m-k-i+j}+\sum_{i=j+1}^kb_ia^{\p{j}}_{m-i+j}\\
&=a^{\p{j}}_m\\
&=q_{2mk+2j}\\
&=q_n,
\end{align*}}
as required.
\end{description}
\end{proof}
\subsection{Notes about the Construction}
In the case where $\seq{a_n}_{n\geq0}$ is the Fibonacci sequence starting from $5$, this construction does not give Ruskey's sequence.  Rather, we obtain the sequence 
\[
\p{5,8,5,4,8,8,8,4,13,8,13,4,21,8,21,4,\ldots}
\]
that eventually satisfies the recurrence
\[
\Mf_a\p{n}=\Mf_a\p{n-\Mf_a\p{n-1}}+\Mf_a\p{n-\Mf_a\p{n-2}}+\Mf_a\p{n-\Mf_a\p{n-3}}.
\]
The Fibonacci numbers each appear twice in this sequence because the Fibonacci recurrence is invariant under rotation (and each rotation of it appears once).

Since any linear recurrent sequence satisfies infinitely many linear recurrences, this construction actually gives infinitely many meta-Fibonacci sequences including a given linear-recurrent sequence.  In addition, the construction can be tweaked in a number of ways to yield slightly different sequences.  For example, one could start from a rotation of the desired sequence.  Or, the initial conditions for the rotations could be chosen differently, since their values are not critical to the construction.  (We only care about the growth rate and recurrent behavior of the rotations.)  But, none of these modifications would suffice to cause our construction to yield Ruskey's sequence, since his sequence has quasi-period $3$ and our construction only yields sequences with even quasi-periods.  This fact seems to indicate that there are many more meta-Fibonacci sequences for a given linear recurrent sequence than our construction can generate.

In our construction, we put two constraints on the $b$ values.  First, we require them to be nonnegative.  With our conventions, it would be impossible to have a solution to a meta-Fibonacci recurrence with infinitely many nonpositive entries.  There are many linear recurrent sequences with positive terms but some negative coefficients.  But, our construction fails for these sequences, since some rotation of such a sequence will have infinitely many nonpositive entries.  Ruskey's question remains open for such sequences.  Second, we require the sum of the $b$ values to be at least $2$.  This was necessary to force the terms of $\seq{a_n}$ to grow superlinearly.  If the $b$ values sum to zero, then they must all be zero, in which case the sequence $\seq{a_n}$ is eventually zero, and, hence, not a sequence of positive integers.  If the $b$ values sum to $1$, then all of them must be zero except for one.  So, the recurrence we obtain is $a_n=b_ia_{n-i}$ for some $i$.  So, in this case, $\seq{a_n}$ is eventually periodic.  Eventually constant sequences eventually satisfy the recurrence $\Mf\p{n}=\Mf\p{n-\Mf\p{n-1}}$, but it is unclear whether higher periods can always be realized within meta-Fibonacci sequences.
\subsection{An Example}
The following example should illustrate most of the nuances of our construction.  Consider the sequence $\seq{a_n}_{n\geq0}$ defined by $a_0=30$, $a_1=40$, $a_2=60$, and $a_n=a_{n-1}+2a_{n-3}$ for $n\geq3$.  (The large initial values allow us to avoid having an unreasonably long initial condition in our meta-Fibonacci sequence.)  The first few terms of this sequence are $\p{30, 40, 60, 120, 200, 320, 560, 960,\ldots}$.  The rotations of $\seq{a_n}$ have the same initial conditions and are given by the following recurrences:
\[\begin{array}{ll}
a^{\p{0}}_n=a^{\p{0}}_{n-1}+2a^{\p{0}}_{n-3} & \p{30, 40, 60, 120, 200, 320, 560, 960,\ldots}\\
a^{\p{1}}_n=a^{\p{1}}_{n-3}+2a^{\p{1}}_{n-2} & \p{30, 40, 60, 110, 160, 280, 430, 720,\ldots}\\
a^{\p{2}}_n=a^{\p{2}}_{n-2}+2a^{\p{2}}_{n-1} & \p{30, 40, 60, 160, 380, 920, 2220, 5360,\ldots}\\
\end{array}
\]
The construction gives the sequence $\seq{q_n}_{n\geq0}$ defined by
\[
\begin{cases}
q_{6m+2j}=a^{\p{j}}_m & 0\leq j<4\\
q_{6m+2j+1}=6\p{3-j} & 0\leq j<4
\end{cases}
\]
as eventually satisfying the meta-Fibonacci recurrence
\begin{align*}
\Mf_a\p{n}=\Mf_a\p{n-\Mf_a\p{n-1}}+\Mf_a\p{n-\Mf_a\p{n-2}}+2\Mf_a\p{n-\Mf_a\p{n-5}}.
\end{align*}
Sure enough, the initial condition
\[
\p{30,18,30,12,30,6,40,18,40,12,40,6,60,18,60,12,60,6}
\]
suffices.  The next term is
\begin{align*}
\Mf_a\p{18}&=\Mf_a\p{18-\Mf_a\p{17}}+\Mf_a\p{18-\Mf_a\p{16}}+2\Mf_a\p{18-\Mf_a\p{13}}\\
&=\Mf_a\p{18-6}+\Mf_a\p{18-60}+2\Mf_a\p{18-18}\\
&=\Mf_a\p{12}+\Mf_a\p{-42}+2\Mf_a\p{0}\\
&=60+0+2\cdot30\\
&=120\\
&=a^{\p{0}}_3,
\end{align*}
as required.  The term after this is
\begin{align*}
\Mf_a\p{19}&=\Mf_a\p{19-\Mf_a\p{18}}+\Mf_a\p{19-\Mf_a\p{17}}+2\Mf_a\p{19-\Mf_a\p{14}}\\
&=\Mf_a\p{19-120}+\Mf_a\p{19-6}+2\Mf_a\p{19-60}\\
&=\Mf_a\p{-101}+\Mf_a\p{13}+2\Mf_a\p{-41}\\
&=0+18+2\cdot0\\
&=18\\
&=6\p{3-0},
\end{align*}
as required.
The rest of the desired terms can continue to be generated this way.

\end{document}